\documentclass[12pt]{amsart}

\usepackage{color}

\usepackage{amsfonts,latexsym,rawfonts,amsmath,amssymb,amsthm,graphicx}
\newtheorem{theorem}{Theorem}[section]
\newtheorem{corollary}[theorem]{Corollary}

\newtheorem{lemma}[theorem]{Lemma}

\newtheorem{definition}{Definition}[section]

\def\bee{\begin{eqnarray}}
\def\beee{\begin{eqnarray*}}
\def\eee{\end{eqnarray}}
\def\eeee{\end{eqnarray*}}

\def\ba{\begin{array}}
\def\ea{\end{array}}

\def\R{\mathbb R}
\def\div{\mathrm{div }}

\numberwithin{equation}{section}

\begin{document}

\title[harmonic maps]
{Global existence of the harmonic map heat flow into Lorentzian manifolds}

\author{ Xiaoli Han, J\"{u}rgen Jost, Lei Liu, Liang Zhao}

\address{Xiaoli Han, Department of Mathematical Sciences, Tsinghua University \\ Beijing 100084, P. R. of China}
\email{xlhan@math.tsinghua.edu.cn}

\address{J\"{u}rgen Jost, Max Planck Institute for Mathematics in the Sciences \\ Inselstrasse 22\\ 04103 Leipzig, Germany}
\email{jjost@mis.mpg.de}

\address{Lei Liu, Max Planck Institute for Mathematics in the Sciences\\ Inselstrasse 22\\ 04103 Leipzig, Germany}
\email{lei.liu@math.uni-freiburg.de}

\address{Liang Zhao, School of Mathematical Sciences,
Beijing Normal University\\
Laboratory of Mathematics and
Complex Systems, Ministry of Education\\ Beijing 100875, P. R. China}
\email{liangzhao@bnu.edu.cn}

\thanks {The authors would like to thank the referee for detailed and useful comments. The research is supported by the National Natural Science Foundation of China, No.11131007, No.11471014, No.11471299, the Fundamental Research Funds for the Central Universities.}


\date{\today}

\maketitle

\begin{abstract}

We investigate a parabolic-elliptic system for maps $(u,v)$ from a compact Riemann surface $M$ into a Lorentzian manifold $N\times{\mathbb{R}}$ with a warped product metric. That system turns the harmonic map type equations into a parabolic system, but keeps the $v$-equation as a nonlinear second order constraint along the flow.  We prove a global existence result of the parabolic-elliptic system by assuming either some geometric conditions on the target Lorentzian manifold or small energy of the initial maps. The result implies the existence of a Lorentzian harmonic map in a given homotopy class with fixed boundary data.\\

R\'esum\'e: Nous \'etudions un syst\`eme parabolique-elliptique pour les applications $(u,v)$ d'une surface  Riemannienne compacte $M$  dans une vari\'e t\'e  lorentzienne $N \times \R$  avec une métrique de produit d\'eform\'ee. Nous transformons les \'equations de type application harmonique en un syst\'eme parabolique, mais conservons l'\'equation de $v$  comme contrainte non lin\'eaire du second ordre le long du flux. Nous d\'emontrons  un r\'esultat global de l'existence du syst\'eme parabolique-elliptique en supposant soit certaines conditions g\'eom\'etriques sur la vari\'e t\'e lorentzienne  ou la petitesse de l'\'energie des valeurs initiales. Le r\'esultat implique l'existence d'une application harmonique lorentzienne  dans une classe d'homotopie donn\'ee avec des donn\'ees a bord fixe.

\end{abstract}

\section{Introduction}

Suppose $(M,h)$ and $(N,g)$ are  compact Riemannian manifolds of dimension $m$ and $n$ respectively. For a map $u\in C^2(M, N)$, the energy functional of $u$ is defined as
\begin{equation}\label{renergy}
E(u)=\frac{1}{2}\int_M |\nabla u|^2dv_h.
\end{equation}
A critical point of the energy functional $E$ in $C^2(M, N)$ is called a harmonic map. By Nash's embedding theorem \cite{Nash}, we can embed $N$ isometrically into some Euclidian space $\R^K$ and the corresponding Euler-Lagrange equation is
$$
\Delta_h u=A(u)(\nabla u,\nabla u),
$$
where $\Delta_h$ is the Laplace-Beltrami operator on $M$ with respect to $h$ and $A$ is the second fundamental form of $N\subset \mathbb{R}^K$. More generally, we define the tension field $\tau (u)$ as
$$
\tau (u)=-\Delta_h u+A(u)(\nabla u,\nabla u).
$$
Thus, $u$ is harmonic if and only if $\tau(u)=0$.

Harmonic maps constitute one of the model problems of geometric analysis and have been widely and systematically studied  for several decades. For example, the methods used in the study of harmonic maps can be adapted to the study of constant mean curvature surfaces, pseudo-holomorphic curves, etc.  In physics, harmonic maps arise as  a mathematical representation of the nonlinear sigma model. This leads to several generalizations. For example, motivated by the supersymmetric sigma model, the map can be coupled with a spinor field, see  \cite{chenjostliwang} and \cite{jostliuzhu} and the references therein.

From another perspective, that of general relativity, it is also natural to replace  the target of the harmonic maps by  Lorentzian manifolds. Recent work on
minimal surfaces in anti-de-Sitter space and their applications in theoretical physics (see e.g. Alday and Maldacena\cite{aldaymaldacena}) shows the importance of this extension. Geometrically, the link between harmonic maps into ${\mathbb{S}}_1^4$ and the conformal Gauss maps of Willmore surfaces in ${\mathbb{S}}^3$ \cite{bryant} also naturally leads to such  harmonic maps. The existence of geodesics in Lorentzian manifolds was studied in \cite{bencifortunatogiannoni}. Variational methods for  such harmonic maps were developed in \cite{greco1} and \cite{greco2}. The regularity of weak solutions was studied in \cite{zhu}, and in \cite{hanjostliuzhao,hanzhaozhu}  energy identities for harmonic map sequences were obtained.

In this paper, we shall address the existence problem for harmonic maps from Riemann surfaces to standard static Lorentzian manifolds (to be defined shortly).

Let us now state our results. Let $M$ be a compact Riemannian manifold with a smooth boundary $\partial M$ and  $N\times{\mathbb{R}}$ be a Lorentzian manifold equipped with a warped product metric of the following form
$$g=g_N-\beta (d\theta)^2,$$
where $({\mathbb{R}}, d\theta^2)$ is the $1$-dimensional Euclidean space, $(N, g_N)$ is an $n$-dimensional compact Riemannian manifold embedded into ${\mathbb{R}^K}$ and $\beta$ is a positive $C^{\infty}$ function on $N$. Since $N$ is compact, there exist two positive constants $\lambda$ and $\Lambda$ such that$$0<\lambda\leq\beta(y)\leq\Lambda<\infty,\ \forall\ y\in N.$$ In fact, a more general form of the warped product metric is
\begin{equation}\label{warpmetric}
g=g_N-\beta(d\theta+\omega)^2,
\end{equation}
where $\omega$ is a $1$-form on $N$. A Lorentzian manifold with a metric of the form \eqref{warpmetric} is called a standard static manifold. For more details on such manifolds, we refer to \cite{KSHM, ONeill}. To simplify the problem, throughout this paper, we assume that $\omega=0$; the case $\omega\neq 0$ will be discussed in future work.

For  $(u,v)\in C^2(M,N\times {\mathbb{R}})$, we consider the following functional
\begin{equation}\label{lag}
E_g(u, v; M)= \frac{1}{2}\int_{M} \left\{|\nabla u|^2- \beta(u)|\nabla v|^2 \right\} dv_h,
\end{equation}
which is called the Lorentzian energy of the map $(u,v)$ on $M$. A critical point $(u,v)$ in $C^2(M,N\times{\mathbb{R}})$ of the functional \eqref{lag} is called a harmonic map from $(M,h)$ into the Lorentzian manifold  $(N\times{\mathbb{R}},g)$. Via direct calculation, one can derive the  Euler-Lagrange equations for \eqref{lag},
\begin{equation}\label{Eleq}
 \left\{
  \begin{array}{cc}
    -\Delta u    = A(u)(\nabla u, \nabla u)-B^{\top}(u)|\nabla v|^2,\ &in\ M\\
     -\div\{\beta(u)\nabla v\} = 0 ,\ &in \ M\\
  \end{array}
\right.
\end{equation}
where $A$ is the second fundamental form of $N$ in ${\mathbb{R}}^K$, $B(u):=(B^1, B^2, \cdots, B^K)$ with
$$
B^j:=-\frac{1}{2}\frac{\partial\beta(u)}{\partial y^j}
$$
and $B^{\top}$ is the tangential part of $B$ along the map $u$. For  details, see \cite{zhu}.

Let us explain  some notations first. For $\Omega\subset M$, we put
$$
E(u,v;\Omega)=\frac{1}{2}\int_{\Omega} \left\{|\nabla u|^2+ |\nabla v|^2 \right\} dx,
$$
$$
E(u;\Omega)=\frac{1}{2}\int_{\Omega} |\nabla u|^2 dx,
$$
and
$$
E(v;\Omega)=\frac{1}{2}\int_{\Omega} |\nabla v|^2 dx.
$$
For brevity, we will omit $\Omega$ if the domain is clear from the context.

\

The existence of harmonic maps in a fixed homotopy class is one of the most fundamental and important problem among those problems related to harmonic maps. When the target manifold is a Riemannian manifold whose sectional curvature is non-positive, Al'ber \cite{alber1,alber2},  Eells-Sampson \cite{EellsSampson} and Hamilton \cite{hamilton} studied a parabolic system, called the harmonic map heat flow, which can be interpreted as a gradient flow for the energy functional \eqref{renergy}, and therefore the energy of the map decreases along the flow. By using the Bochner formula technique, they obtained a global regular solution of the harmonic map flow that asymptotically converges to a  harmonic map in the same homotopy class. When the target manifold is a general Riemannian manifold, the harmonic map flow may develop singularities in finite time even in dimension two \cite{Chang-Ding-Ye}. Nevertheless, a  global weak solution could be constructed in \cite{St}.

As in the Riemannian case, it is also desirable to find the conditions for the existence of a  harmonic map in a given homotopy class when the target manifold is Lorentzian. A key difficulty arises from the fact that the functional $E_g$ is not bounded from below. Therefore, classical variational approaches developed for harmonic maps cannot be applied to study the existence of Lorentzian harmonic maps. In this case, it seems natural to seek an analogous parabolic system such as the harmonic map heat flow, which puts a time derivative on the left hand side of \eqref{Eleq}. By direct computations, one can find that although the flow is still a gradient flow, due to the unboundedness of the energy $E(u,v)$ caused by the Lorentzian metric, we can no longer expect the long time existence of the flow. One can also see that the  Bochner formula technique does not work in the  Lorentzian case, because in that case, the positivity of a certain term gets lost. In particular, because of the lack of boundedness,  even  the blow up analysis along the heat flow cannot be carried over from the Riemannian case. We therefore need to apply other  methods to deal with the existence problem of Lorentzian harmonic maps in a given homotopy class. % As far as the authors know,  there is no such kind of results for the Lorentzian case till now.

The natural idea is to disentangle the two contributions to the energy \eqref{lag}. The term $-\int \beta(u)|\nabla v|^2$, while having the wrong sign is easier because $v$ is scalar valued. Therefore, one could carry its Euler-Langrange equation $\div\{\beta(u)\nabla v\} = 0$ along as a constraint while trying to minimize the other term which depends on $u$ only. We then have two options for handling the latter, either by a variational scheme or by a parabolic method. The first method was pioneered by Hardt-Kinderlehrer-Lin \cite{H-K-L} in the context of liquid crystals. We shall briefly outline below how this method should naturally work in our context. Here, however, we apply the second method, that is, consider a heat flow for $u$ with the $u$-dependent elliptic equation for $v$ as a constraint. That method is inspired by the construction of  \cite{chenjostsunzhu,jostliuzhu} for studying the flow of Dirac harmonic maps. In the end, it seems that the technical difficulties for either method are similar, and both should find further useful applications on such coupled problems.

% To overcome these difficulties, we consider a parabolic-elliptic flow inspired by the construction of  \cite{chenjostsunzhu,jostliuzhu} to study the flow of Dirac harmonic maps. This is the key idea of this paper.   We convert the equation for $u$ into a parabolic equation, but keep  the equation for  $v$ elliptic. We thus carry it as a nonlinear constraint along the flow. Since the flow (and in fact, already the original elliptic system) couples  $u$ and $v$, $v$ also depends on time implicitly.

When studying either scheme,  new difficulties arise from the Lorentzian metric and the particular coupling structure.  As mentioned, we here work with the  parabolic-elliptic system. Also, the system with boundary condition is more complicated than a heat flow without boundary. After  a subtle and careful analysis, we can nevertheless handle this  parabolic-elliptic system. Our results include short time existence, a small energy regularity theorem, the blow-up behaviour near a singularity and a global existence result. As an application, we will prove the existence of a Lorentzian harmonic map in a fixed homotopy class under several  conditions.

\

We now introduce our parabolic-elliptic system. Let $\phi\in C^{2+\alpha}(M,N)$, $\psi\in C^{2+\alpha}(\partial M,\R)$ for some $0<\alpha<1$. Consider the flow
\begin{align}\label{heateq}
\begin{cases}
\partial_t u=\Delta u+ A(u)(\nabla u,\nabla u)-B^\top(u)|\nabla v|^2,\ &in\ M\times [0,T) \\
-\div (\beta(u)\nabla v)=0, \ &in\ M\times [0,T)
\end{cases}
\end{align}
with the boundary-initial data
\begin{align}\label{boundary-data}
\begin{cases}
u(x,t)=\phi(x), \  &on\  M\times\{t=0\}\cup\partial M\times\{t>0\},\\
v(x,t)=\psi(x),\  &on \  \partial M\times\{t>0\}.
\end{cases}
\end{align}
By  standard elliptic theory, for the above $(\phi,\psi)$, there exists a unique solution $v\in C^{2+\alpha}(M)$ of the equation
\begin{align}\label{equation:07}
\begin{cases}
-\div (\beta(\phi)\nabla v)=0\ &in\ M,\\
v(x)=\psi\ &on \ \partial M.
\end{cases}
\end{align}
This $v$ is called an extension of $\psi$. For simplicity, we still denote it by $\psi\in C^{2+\alpha}(M)$ and in the following, we use the extension when  needed.

\

Now, we state our first main result, concerning the short time existence of the flow \eqref{heateq}.

\begin{theorem}\label{thm:shortime-existence}
Let $(M^m,h)$ $(m= 2)$ be a compact Riemann surface with a smooth boundary $\partial M$ and $(N,g)$ be another compact Riemannian manifold. Then for any $$\phi\in C^{2+\alpha}(M,N),\ \psi\in C^{2+\alpha}(\partial M, \R)$$ where $0<\alpha<1$, the problem \eqref{heateq} and \eqref{boundary-data} admits a unique solution $$u\in \cap_{0<s<T_1} C^{2+\alpha,1+\alpha/2}(M\times [0,s]),$$ and
$$v,\nabla v\in \cap_{0<s<T_1} C^{\alpha,\alpha/2}(M\times [0,s]),\ v\in L^\infty([0,T_1);C^{2+\alpha}(M)),$$ for some time $T_1>0$. Here, the maximum existence time $T_1$ is characterized by the condition that
$$\limsup_{x\in M,t\to T_1}E(u;B^M_r(x))>\epsilon_1\mbox{ for any } r>0,$$
where $\epsilon_1$ is the constant in Lemma \ref{lem:small-energy-regularity} and $B^M_r(x)$ is a geodesic ball in $M$ . Moreover, the set
\begin{align}\label{def:singular-set}
S(u,T_1):=\{x\in M|\limsup_{t\to T_1}E(u;B^M_r(x))>\epsilon_1\mbox{ for any } r>0\}
\end{align}
is finite and a point in it is called a singularity at the singular time $T_1$.
\end{theorem}

Moreover, we show that at a singular point $(x, T_1)$, $0<T_1\leq\infty$, after suitable space-time rescalings,  a
nontrivial harmonic sphere splits off.

\begin{theorem}\label{thm:blow-up}
Let $(u,v)$ be the solution to \eqref{heateq} with the boundary-initial data \eqref{boundary-data} in Theorem \ref{thm:shortime-existence}. Suppose $(x_0,T_1)$ is a singularity such that
\begin{equation}
\limsup_{\substack{t\nearrow T_1}}E(u(t);B_r^M(x_0))>\epsilon_1\quad \mbox{for all}\quad r>0.
\end{equation}
Then
\begin{itemize}
\item[(1)] if $x_0\in M\setminus \partial M$, there exist sequences $t_i\nearrow T_1$, $x_i\to x_0\in M$, $r_i\to 0$ and a nontrivial harmonic map $\widetilde{u}:\R^2\to N$, such that as $i\to \infty$,
\begin{align*}
u(x_i+r_ix,t_i)&\to \widetilde{u}(x) \quad in \quad C^1_{loc}(\R^2).
\end{align*}
$\widetilde{u}$ has finite energy and conformally extends to a smooth harmonic sphere.
\item[(2)] if $x_0\in\partial M$, we have $\frac{dist(x_i,\partial M)}{r_i}\to \infty$ and the same bubbling statement as in $(1)$ holds.
\end{itemize}
\end{theorem}

\

Next we present two main results which establish  the existence of a harmonic map from $M$ into the Lorentzian manifold $N\times \R$ in any given homotopy class in two cases. In the first case, we assume that the initial energy $E(\phi;M)$ is small and in the second case, we assume that the sectional curvature $K_N$ of the Riemannian manifold $N$ is non-positive. More precisely, we have
\begin{theorem}\label{thm:small-initial-energy}
For any given
$(\phi,\psi)\in C^{2+\alpha}(M,N\times\R)$, there exist constants $\overline{\epsilon}_1\geq 0$, $\overline{\epsilon}_2>0$ and $\overline{\epsilon}>0$ which are defined by $$\overline{\epsilon}_1:=\inf\{E(w)|w\in W^{1,2}(M,N),w|_{\partial M}=\phi\},$$ $$\overline{\epsilon}_2:=\inf\{E(w)|w:S^2\to N\mbox{ is a harmonic map and nonconstant}\},$$ and $\overline{\epsilon}:=\overline{\epsilon}_1+\overline{\epsilon}_2>0$ such that if $$E(\phi;M)+(\Lambda-\lambda)E(\psi;M)\leq \overline{\epsilon},$$ then the parabolic-elliptic system \eqref{heateq} and \eqref{boundary-data} admits a global solution $$u\in \cap_{0<s<\infty} C^{2+\alpha,1+\alpha/2}(M\times [0,s]),$$ and
$$v,\nabla v\in \cap_{0<s<\infty} C^{\alpha,\alpha/2}(M\times [0,s]),\ v\in L^\infty([0,\infty);C^{2+\alpha}(M)).$$ Moreover, $(u(x,t),v(x,t))$ subconverges in $C^2$ to a
harmonic map $(u_\infty, v_\infty):M\to N\times \R$ with boundary data $u_\infty|_{\partial M}=\phi$ and $v_\infty|_{\partial M}=\psi$.
\end{theorem}

Theorem \ref{thm:small-initial-energy} generalizes the result for the harmonic map heat flow by Chang \cite{Chang} to the Lorentzian case. Before introducing the theorem in the second case, we need the following definition.
\begin{definition}
Let $\rho\in C^2(N)$ be a nonnegative function on a Riemannian manifold $(N,g)$ and $d(x,x_0)$ be the distance between $x\in N$ and $x_0\in N$. If $\rho$ satisfies
\begin{enumerate}
   \item $\nabla^2 \rho(x)>0$ for any $x\in N$;
   \item $\rho(x)\leq C(1+d(x,x_0))^{d_0}$ for some positive integer $d_0$ and fixed $x_0 \in N$,
\end{enumerate}
we call $\rho$ a nonnegative strictly convex function with polynomial growth.
\end{definition}
When $N$ has non-positive sectional curvature, then the squared distance function $d^2(.,x_0)$ for any $x_0 \in \widetilde{N}$ is such a function on $\widetilde{N}$, where $\widetilde{N}$ is the universal covering space of $(N, g)$, with metric $\widetilde{g}:=\pi_N^*g$ being  the pull-back metric on $\widetilde{N}$ and $\pi_N:\widetilde{N} \to N$ being the projection. Therefore, our subsequent results will apply to targets of non-positive sectional curvature.

We have
\begin{theorem}\label{thm:convex-func-target-mfd}
Suppose the universal covering space $(\widetilde{N}, \widetilde{g})$ admits a nonnegative strictly convex function $\rho\in C^2(\widetilde{N})$ with polynomial growth. For any given
$\phi\in C^{2+\alpha}(M,N),\ \psi\in C^{2+\alpha}(\partial M,\R)$, the parabolic-elliptic system \eqref{heateq} and \eqref{boundary-data} admits a global solution $$u\in \cap_{0<s<\infty} C^{2+\alpha,1+\alpha/2}(M\times [0,s]),$$ and
$$v,\nabla v\in \cap_{0<s<\infty} C^{\alpha,\alpha/2}(M\times [0,s]),\ v\in L^\infty([0,\infty);C^{2+\alpha}(M)).$$ Moreover, $(u(x,t),v(x,t))$ subconverges in $C^2$ to a
harmonic map $(u_\infty, v_\infty):M\to N\times \R$ with boundary data $u_\infty|_{\partial M}=\phi$ and $v_\infty|_{\partial M}=\psi$.
\end{theorem}

In the Riemannian case, such a result was first proved by Ding and Lin \cite{dinglin} when the universal covering of the target manifold admits a nonnegative strictly convex function with quadratic growth. The polynomial growth case was proved by Li-Zhu \cite{lizhu} and Li-Yang \cite{liyang}.

As already mentioned, since for a Riemannian manifold $N$ with non-positive sectional curvature $K_N$,  the square of the distance function on the universal covering of $N$ is a nonnegative strictly convex function with quadratic growth, the existence theorem for harmonic maps by Al'ber \cite{alber1,alber2}, Eells and Sampson \cite{EellsSampson}, Hamilton \cite{hamilton} and Hildebrandt-Kaul-Widman \cite{Hildebrandt75} can be generalized for two-dimensional domains to the Lorentzian case as a corollary of Theorem \ref{thm:convex-func-target-mfd}.
\begin{theorem}\label{thm:non-positive-curv}
When  $(N,g)$ is a compact Riemannian manifold with non-positive sectional curvature,  the conclusions in Theorem \ref{thm:convex-func-target-mfd}  hold.
\end{theorem}

The paper is organized as follows. In Section 2, we derive some a priori estimates. In Section 3, we prove a small energy regularity lemma. Also, we establish the short time existence theorem \ref{thm:shortime-existence} and give a characterization of the singularities in this section. In section 4, we analyze the blow up behavior of the singularities developed by the flow and prove our Theorem \ref{thm:blow-up}. In section 5, we use the blow up analysis to get some long time existence and convergence results. Theorem \ref{thm:small-initial-energy} and Theorem \ref{thm:convex-func-target-mfd} are proved in this section.  In the final section, we shall briefly discuss the method of \cite{H-K-L}.

\

We would like to thank the referee for pointing  \cite{H-K-L} out to us.

\

\noindent{\textbf{Notation:}}
\begin{align*}
\mathcal{V}(M_s^t;N\times \R):=\{(u,v):&M\times [s,t)\to N\times\R,\ v\in L^\infty([s,t);C^{2+\alpha}(M)),\\ &v,\nabla v\in \cap_{s<\rho<t} C^{\alpha,\alpha/2}(M\times [s,\rho]),\\ &u\in \cap_{s<\rho<t} C^{2+\alpha,1+\alpha/2}(M\times [s,\rho]).\}
\end{align*}
Throughout this paper, we use $C$ to denote a universal constant.

\

\section{Some a priori estimates}

\

First, we present a lemma which ensures that the Lorentzian energy $E_g$ is non-increasing along the flow \eqref{heateq}. This is an important property of our  parabolic-elliptic flow.

\begin{lemma}\label{energy}
Suppose $(u,v)\in \mathcal{V}(M_0^{T_1};N\times \R)$ is a solution of (\ref{heateq}) and \eqref{boundary-data}, then the Lorentzian energy $E_g(u(t), v(t))$ is non-increasing on $[0,T_1)$ and for any $0\leq s\leq t<T_1$, there holds $$E_g(u(t), v(t))+\int_s^t\int_M|\partial_tu|^2dxdt\leq E_g(u(s), v(s)).$$
\end{lemma}
\noindent{\it Proof.} First, we may assume $v,\nabla v,\nabla u\in C^1(M\times (0,T_1))$. By direct computations, we get
\begin{eqnarray*}
\frac{\partial}{\partial t} E_g(u, v)
&=&\int_{M}\nabla u\cdot\nabla u_t+\int_{M}B^\top(u)\cdot u_t|\nabla v|^2-\int_{M}\beta(u)\nabla v\cdot\nabla v_t\\
&=&\int_{\partial M}\frac{\partial u}{\partial n}u_t-\int_{M}\Delta u u_t+\int_{M}B^\top(u)\cdot u_t|\nabla v|^2-\int_{M}\beta(u)\nabla v\cdot\nabla v_t\\
(u_t|_{\partial M}=0)
&=&-\int_{M} (u_t+B^{\top}(u)|\nabla v|^2))^{\top}\cdot u_t+\int_{M}B^\top(u)\cdot u_t|\nabla v|^2\\
&&-\int_{M}\beta(u)\nabla v\cdot\nabla v_t\\
&=&-\int_{M}|u_t|^2-\int_{M}\beta(u)\nabla v\cdot\nabla v_t\\
(v_t|_{\partial\Omega}=0)
&=&-\int_{M}|u_t|^2+\int_{M}\div(\beta(u)\nabla v)v_t\\
&=&-\int_{M}|u_t|^2\leq 0.
\end{eqnarray*}
For the general case that $(u,v)\in \mathcal{V}(M_0^{T_1};N\times \R)$, this can be done just by replacing the derivative $\frac{\partial}{\partial t}$ by difference quotients in the proof.
Then the conclusion of the lemma follows immediately.
\hfill{$\square$}

\

The next lemma tells us that the $L^2$ norms (energy) of $u$ and $v$ are always bounded by the initial data.
\begin{lemma}\label{lem:Dirichlet-energy}
Suppose $(u,v)\in \mathcal{V}(M_0^{T_1};N\times \R)$ is a solution of (\ref{heateq}) and \eqref{boundary-data}, then for any $0\leq t<T_1$, there holds
\begin{eqnarray*}
\int_{M}|\nabla u|^2(\cdot,t)dx&\leq& \int_{M}|\nabla \phi|^2dx+(\Lambda-\lambda)\int_{M}|\nabla \psi|^2dx \mbox{ and }\\ \int_{M}|\nabla v|^2(\cdot,t)dx&\leq& \frac{\Lambda}{\lambda}\int_{M}|\nabla \psi|^2dx.
\end{eqnarray*}
\end{lemma}

\noindent{\it Proof.} Multiplying the equation of $v$ by $v-\psi$ and integrating on $M$, we get
\begin{eqnarray}\label{estimateforv1}
0&=&\int_{M}\div \{\beta(u)\nabla v\}(v-\psi)dx\nonumber\\
&=&\int_{\partial M}\beta(u)\frac{\partial v}{\partial n}(v-\psi)dx
-\int_{M}\beta(u)\nabla v \cdot\nabla(v-\psi)dx\nonumber\\
&=&-\int_{M}\beta(u)|\nabla v|^2dx+\int_{M}\beta(u)\nabla v\cdot \nabla \psi dx,
\end{eqnarray}
where in the last equality we use the fact that $v|_{\partial \Omega}=\psi$.

By Young's inequality, we have
\begin{eqnarray}\label{estimateforv2}
\int_M \beta(u)|\nabla v|^2 dx&\leq& \int_{M}\beta(u)|\nabla v \cdot \nabla \psi| dx\nonumber\\
&\leq&\frac{1}{2}\int_{M}\beta(u)|\nabla v|^2 dx+\frac{1}{2}\int_{M}\beta(u)|\nabla \psi|^2 dx.
\end{eqnarray}
Thus we obtain
\begin{equation}\label{estimateforv3}
\int_M \beta(u)|\nabla v|^2 dx
\leq\int_{M}\beta(u)|\nabla \psi|^2 dx,
\end{equation}
and
\begin{equation}
\int_M |\nabla v|^2 dx
\leq\frac{\Lambda}{\lambda}\int_{M}|\nabla \psi|^2 dx.
\end{equation}
Combining \eqref{estimateforv3} with Lemma \ref{energy}, we have
\begin{eqnarray*}
\frac{1}{2}\int_{M}|\nabla u|^2dx&\leq& E_g(u,v)+\frac{1}{2}\int_M \beta(u)|\nabla v|^2 dx\\
&\leq& E_g(\phi,\psi)+\frac{1}{2}\int_M \beta(u)|\nabla \psi|^2 dx\\
&\leq& E(\phi)+(\Lambda-\lambda) E(\psi).
\end{eqnarray*}
\hfill{$\square$}

As a direct corollary of the above lemma, we have
\begin{corollary}\label{cor:energy-t}
Suppose $(u,v)\in \mathcal{V}(M_0^{T_1};N\times \R)$ is a solution of (\ref{heateq}) and \eqref{boundary-data}, then
$$\int_0^{T_1}\int_{M}|u_t|^2dxdt\leq (1+\Lambda-\lambda)E(\phi,\psi).$$
\end{corollary}

\noindent{\it Proof.} By Lemma \ref{energy} and \eqref{estimateforv3}, we know that
\begin{eqnarray*}
\int_0^t\int_{M}|u_t|^2dxdt&\leq& E_g(u(\cdot, 0), v(\cdot, 0))-E_g(u(\cdot, t), v(\cdot, t))\\
&\leq& \frac{1}{2}\int_{M}|\nabla \phi|^2dx-\frac{\lambda}{2}\int_{M}|\nabla \psi|^2dx+\frac{1}{2}\int_{M}\beta(u)|\nabla v(\cdot, t)|^2dx\\
&\leq & \frac{1}{2}\int_{M}|\nabla \phi|^2dx+\frac{\Lambda-\lambda}{2}\int_{M}|\nabla \psi|^2\leq (1+\Lambda-\lambda)E(\phi,\psi).
\end{eqnarray*}
\hfill{$\square$}

\

In Lemma \ref{lem:Dirichlet-energy}, we prove that $\|\nabla v\|_{L^2(M)}$ is uniformly bounded by using an integration method. In fact, we can use the theory of second order elliptic equations of divergence form to obtain a stronger $W^{1,p}$ estimate for $v$ along the flow. More precisely, we have

\begin{lemma}\label{lem:Lp-estimates-v} ($W^{1, p}$ estimate for $v$)
Suppose $(u,v)\in \mathcal{V}(M_0^{T_1};N\times \R)$ is a solution of (\ref{heateq}) and \eqref{boundary-data}, then for any $p>1$, $0\leq t<T_1$, we have
\begin{eqnarray*}
\int_{M}|\nabla v|^p(\cdot,t)dx\leq C\int_{M}|\nabla \psi|^pdx,
\end{eqnarray*}
where $C$ only depends on $p,M,\lambda,\Lambda$.
\end{lemma}

\noindent{\it Proof.} Set $\tilde v=v-\psi$,  and we have $\tilde v=0$ on $\partial M$. Since $v$ satisfies a second order elliptic divergence equations, then
$$div (\beta(u)\nabla\tilde v)=-div (\beta(u)\nabla\psi).$$ Thus, by Theorem $1$ in \cite{Me}, we know that
$$\int_{M}|\nabla\tilde v|^pdx\leq C \int_{M}|\beta(u)\nabla\psi|^pdx,$$
where $C$ only depends on $p,M,\Lambda,\lambda$. It implies that
\begin{eqnarray*}
\int_{M}|\nabla v|^pdx\leq C \int_{M}|\nabla \psi|^pdx.
\end{eqnarray*}
\hfill{$\square$}

\begin{lemma}\label{lem:two-balls}
Let $(u,v)\in \mathcal{V}(M_0^{T_1};N\times \R)$ be a solution to (\ref{heateq}) and \eqref{boundary-data}. There exists a positive constant $R_0<1$ such that, for any $x_0\in M$, $0\leq R\leq R_0$ and $0<s\leq t<T_0$, there holds
\begin{eqnarray}\label{inequality:13}
E(u(t);B^M_R(x_0))\leq E(u(s);B^M_{2R}(x_0))+C_1\frac{t-s}{R^2}+C_2(t-s),
\end{eqnarray}
where $C_1$ and $C_2$ depend on $\lambda,\Lambda,M,N,E(\phi),\|\psi\|_{W^{1,4}(M)}$.
%and
%\begin{eqnarray}\label{inequality:14}
%E(u(s);B^M_R(x_0))\leq E(u(t);B^M_{2R}(x_0))+C\frac{t-s}{R^2}+C\int_s^t\int_{M}|\partial_tu|^2dxdt.
%\end{eqnarray}
\end{lemma}

\noindent{\it Proof.} Let $\eta\in C^\infty_0(B^M_{2R}(x_0))$ be a cut-off function such that $\eta(x)=\eta(|x-x_0|)$, $0\leq\eta\leq1$, $\eta|_{B^M_{R}(x_0)}\equiv 1$ and $|\nabla\eta|\leq\frac{C}{R}$. By direct computations, we get
\begin{eqnarray*}
\frac{d}{dt}\frac{1}{2}\int_{M}|\nabla u|^2\eta^2 &=&\int_{M}\langle\nabla u,\nabla u_t\rangle\eta^2\\
&=&\int_{\partial B^M_{2R}(x_0)}\frac{\partial u}{\partial n}\cdot u_t\eta^2-\int_{M}\langle\Delta u, u_t\rangle\eta^2-2\int_{M}\nabla u\cdot\nabla\eta\eta u_t\\
&=&\int_{M}\langle -u_t-B^{\top}(u)|\nabla v|^2, u_t\rangle\eta^2-2\int_{M}\nabla u\cdot\nabla\eta\eta u_t\\
&=&-\int_{M}|u_t|^2\eta^2-\int_MB^\top(u)|\nabla v|^2\cdot u_t\eta^2-2\int_{M}\nabla u\cdot\nabla\eta\eta u_t.
\end{eqnarray*}
By Lemma \ref{lem:Dirichlet-energy}, Lemma \ref{lem:Lp-estimates-v} and Young's inequality, we have
\begin{align*}
\frac{d}{dt}\frac{1}{2}\int_{M}|\nabla u|^2\eta^2
&\leq-\frac{1}{2}\int_{M}|u_t|^2\eta^2+C\int_{M}|\nabla u|^2|\nabla \eta|^2+C\int_{M}|\nabla v|^4\eta^2\\
&\leq \frac{C_1}{R^2}+C_2.
\end{align*}
By integrating the above inequality from $s$ to $t$, we can get \eqref{inequality:13}.

%On the other hand, by Lemma \ref{lem:Dirichlet-energy}, Lemma \ref{lem:Lp-estimates-v} and Young's inequality, we also have
%\begin{align*}
%\frac{d}{dt}\frac{1}{2}\int_{M}|\nabla u|^2\eta^2
%&\geq-\frac{3}{2}\int_{M}|u_t|^2\eta^2-C\int_{M}|\nabla u|^2|\nabla \eta|^2-C\int_{M}|\nabla v|^4\eta^2\\
%&\geq-\frac{3}{2}\int_{M}|u_t|^2\eta^2- \frac{C}{R^2}.
%\end{align*}
%Then \eqref{inequality:14} follows immediately from integrating the above inequality from $s$ to $t$.
\hfill{$\square$}

Next, we derive an $\epsilon_1$- regularity lemma.

\begin{lemma}\label{lem:small-energy-regularity}
Let $(\phi,\psi)\in C^{2+\alpha}(M,N\times \R)$, $z_0=(x_0,t_0)\in M\times (0, T_1]$, denote $P_{r}^M(z_0):=B^M_r(x_0)\times [t_0-r^2,t_0]$. Assume that
$(u,v)\in \mathcal{V}(M_0^{T_1};N\times \R)$, then there exist two positive constants
 $\epsilon_1=\epsilon_1(M,N, \|\phi\|_{C^{2+\alpha}( M)},\|\psi\|_{C^{2+\alpha}( M)})>0$ and
$C=C(\alpha,r,M,N,\|\phi\|_{C^{2+\alpha}( M)},\|\psi\|_{C^{2+\alpha}( M)})>0$ such that if
\[
\sup_{[t_0-4r^2,t_0]}E(u(t),B^M_{2r}(x_0))\leq\epsilon_1,
\]
we have
\begin{equation}\label{inequality:01}
r\|\nabla v\|_{L^\infty(P_{r}^M(z_0))}+
r\|\nabla u\|_{L^\infty(P_{r}^M(z_0))}\leq C
\end{equation}
and for any $0<\beta<1$,
\begin{equation}\label{inequality:02}
\sup_{t_0-\frac{r^2}{4}\leq t\leq t_0}\|v(t)\|_{C^{2+\alpha}(B_{r/2}^M(x_0))}+
\|u\|_{C^{\beta,\beta/2}(P_{r/2}^M(z_0))}+\|\nabla u\|_{C^{\beta,\beta/2}(P_{r/2}^M(z_0))}\leq C(\beta),
\end{equation}

Moreover, if
\[
\sup_{x_0\in M}\sup_{[t_0-r^2,t_0]}E(u(t),B^M_r(x_0))\leq\epsilon_1,
\]
then
\begin{equation}\label{inequality:03}
\sup_{t_0-\frac{r^2}{8}\leq t\leq t_0}\|v(t)\|_{C^{2+\alpha}(M)}+
\|u\|_{C^{2+\alpha,1+\alpha/2}(M\times [t_0-\frac{r^2}{8},t_0])}\leq C,
\end{equation}
and
\begin{equation}\label{inequality:030}
\| v\|_{C^{\alpha,\alpha/2}(M\times [t_0-\frac{r^2}{8},t_0])}+\|\nabla v\|_{C^{\alpha,\alpha/2}(M\times [t_0-\frac{r^2}{8},t_0])}\leq C.
\end{equation}
\end{lemma}

\begin{proof}
\textbf{Step 1:} We prove \eqref{inequality:02}, \eqref{inequality:03} and \eqref{inequality:030} under the assumption that \eqref{inequality:01} is true.

Taking the cut-off function $\eta\in C_0^\infty(P_{r}^M(z_0))$ such that $0\leq\eta\leq 1$, $\eta|_{P_{3r/4}^M(0)}\equiv 1$, $|\nabla^j\eta|\leq \frac{C}{r^j},j=1,2$ and $|\partial_t\eta|\leq \frac{C}{r^2}$, set $U=\eta u$, then
\begin{align*}
\begin{cases}
\partial_tU-\Delta U=f, \quad &in\quad P_{r}^M(z_0);\\
U(x,t)=0,   \quad &on\quad B_{r}^M(x_0)\times \{t=t_0-r^2\};\\
U(x,t)=\eta\varphi,   \quad &on\quad \partial B^M_r(x_0)\times (t_0-r^2,t_0),
\end{cases}
\end{align*}
where \[
f:=\eta(\partial_t-\Delta)u+ u(\partial_t-\Delta)\eta -2\nabla\eta\nabla u.
\]
By the standard parabolic theory, for any $1<p<\infty$, we have
\begin{align*}
\|U\|_{W^{2,1}_p(P_{r}^M(z_0))}&\leq C\big(\|f\|_{L^p(P_{r}^M(z_0))}+\|\eta\varphi\|_{W^{2,1}_p(P_{r}^M(z_0))}+\|U\|_{L^p(P_{r}^M(z_0))}\big)\\
&\leq
C\big(1+\|\varphi\|_{C^2( M)}\big),
\end{align*}
where we use the fact that $f\in L^\infty$ under the equation \eqref{heateq} and assumption \eqref{inequality:01}. Then, by Sobolev's embedding, for any $0<\beta=1-4/p<1$, we obtain
\begin{align}\label{inequality:05}
\| u\|_{C^{\beta,\beta/2}(P_{3r/4}^M(z_0))}+\|\nabla u\|_{C^{\beta,\beta/2}(P_{3r/4}^M(z_0))}&\leq
\| U\|_{C^{\beta,\beta/2}(P_{r}^M(z_0))}+\| \nabla U\|_{C^{\beta,\beta/2}(P_{r}^M(z_0))}\notag\\&\leq C\|U\|_{W^{2,1}_p(P_{r}^M(z_0))}\leq
C(\beta)(1+\|\varphi\|_{C^2(M)}).
\end{align}

Taking a cut-off function $\xi(x)=\xi(|x-x_0|)\in C^\infty_0(B^M_r(x_0))$ such that $0\leq\xi\leq 1$, $\xi|_{B^M_{3r/4}}\equiv 1$ and $|\nabla^j\xi|\leq\frac{C}{r^j}$, $j=1,2$, set $V=\xi v$, then we have
\begin{align*}
\begin{cases}
\Delta V=h, & \mbox{ in } B^M_r(x_0); \\
  V=\xi \psi, & \mbox{ on } \partial B^M_r(x_0),
\end{cases}
\end{align*}
where $h=\Delta\xi v+2\nabla\xi\nabla v+\xi\Delta v\in L^\infty$. By the standard elliptic estimates and Sobolev embedding, we get
\begin{align}\label{inequality:04}
\|v\|_{C^{1,1-2/p}(B^M_{3r/4}(X_0))}\leq C\|V\|_{W^{2,p}(B^M_r(x_0))}\leq C(1+\|\psi\|_{C^2(M)})
\end{align}
for any $2<p<\infty$. Noting that \eqref{inequality:05} and \eqref{inequality:04} yields $\Delta v\in C^{\alpha}(B^M_{3r/4}(x_0))$, by the Schauder estimates and taking some suitable cut-off functions as before, we get
\begin{align}\label{inequality:06}
\|v\|_{C^{2+\alpha}(B^M_{r/2}(x_0))}\leq C(1+\|\phi\|_{C^2(M)})(1+\|\psi\|_{C^{2+\alpha}(M)})
\end{align}
for any $t_0-\frac{r^2}{4}\leq t\leq t_0$. Then \eqref{inequality:02} follows from \eqref{inequality:05} and \eqref{inequality:06} immediately.

To prove \eqref{inequality:03} and \eqref{inequality:030}, we rewrite the equation of $v$ as follows
$$\Delta v=\gamma(u)\nabla u\nabla v,$$
where $\gamma(u)=\frac{2B^\top(u)}{\beta(u)}$. Then for any $t_0-\frac{r^2}{4}<t,s<t_0$, we have
\begin{align*}
  \Delta (v(\cdot,t)-v(\cdot,s))=&\gamma(u(\cdot,t))\nabla u(\cdot,t)\nabla (v(\cdot,t)-v(\cdot,s))\\&+\left(\gamma(u(\cdot,t))\nabla u(\cdot,t)-\gamma(u(\cdot,s))\nabla u(\cdot,s)\right)\nabla v(\cdot,s) \ in\ M.
\end{align*}
Combining \eqref{inequality:01}, \eqref{inequality:05} with the fact that $v(\cdot,t)-v(\cdot,s)=0 \ on\ \partial M$, by the standard elliptic estimates and Sobolev embedding, we obtain
\begin{align*}
\|v(\cdot,t)-v(\cdot,s)\|_{C^{1+\alpha}(M)}\leq C\|\gamma(u(\cdot,t))\nabla u(\cdot,t)-\gamma(u(\cdot,s))\nabla u(\cdot,s)\|_{L^{\infty}(M)}\leq C|s-t|^{\alpha/2}.
\end{align*}
Thus, we get $\|\nabla v\|_{C^{\alpha,\alpha/2}(M\times [t_0-\frac{r^2}{4},t_0])}+\|\nabla v\|_{C^{\alpha,\alpha/2}(M\times [t_0-\frac{r^2}{4},t_0])}\leq C$ which is \eqref{inequality:030} and
\begin{align*}
\begin{cases}
  \partial_t u-\Delta u\in C^{\alpha,\frac{\alpha}{2}}(M\times [t_0-\frac{r^2}{4},t_0]), \\
  u|_{\partial M}=\phi\in C^{2+\alpha}(M).
\end{cases}
\end{align*}
Taking some suitable cut-off function and by the standard Schauder estimates for parabolic equations, we have $u\in C^{2+\alpha,1+\alpha/2}(M\times [t_0-\frac{r^2}{8},t_0])$ and
\begin{align*}
&\|u\|_{C^{2+\alpha,1+\alpha/2}(M\times [t_0-\frac{r^2}{8},t_0])}\\&\leq C(\|\partial_t u-\Delta u\|_{C^{\alpha,\frac{\alpha}{2}}(M\times [t_0-\frac{r^2}{4},t_0])}+\|u\|_{C^{0}(M\times [t_0-\frac{r^2}{4},t_0])}+\|\phi\|_{C^{2+\alpha}(M)})\leq C.
\end{align*}
Thus we get \eqref{inequality:03}.

\

\noindent\textbf{Step 2:} Next we prove \eqref{inequality:01}. The idea is similar as in \cite{LW, Schoen}. Without loss of generality, we may assume $r=\frac{1}{2}$. Choose $0\leq\rho<1$ such that
\[
(1-\rho)^2\sup_{P^M_\rho(z_0)}|\nabla u|^2=\max_{0\leq\sigma\leq 1}\{(1-\sigma)^2\sup_{P^M_\sigma(z_0)}|\nabla u|^2\}
\]
and choose $z_1=(x_1,t_1)\in P^M_\rho(z_0)$ such that
\[
|\nabla u|^2(z_1)=\sup_{P^M_\rho(z_0)}|\nabla u|^2:=e.
\]
We claim that $$(1-\rho)^2e\leq 4.$$

We proceed  by contradiction. If $(1-\rho)^2e>4$, we set
\begin{align*}
\widetilde{u}(x,t):=u(x_1+e^{-\frac{1}{2}}x,t_1+e^{-1}t)\quad and \quad \widetilde{v}(x):=v(x_1+e^{-\frac{1}{2}}x,t_1+e^{-1}t).
\end{align*}
Denoting
\[
D_r(0):=\{x\in B_r(0)|x_1+e^{-\frac{1}{2}}x\in B^M_1(x_0)\}
\]
and
\[
S_r:=\{(x,t)\in B_r(0)\times [-r^2,0]|(x_1+e^{-\frac{1}{2}}x,t_1+e^{-1}t)\in P^M_1(z_0)\},
\]
then
\begin{align}
\begin{cases}
\partial_t \widetilde{u}=\Delta \widetilde{u}+ A(\widetilde{u})(\nabla \widetilde{u},\nabla \widetilde{u})-B^\top(\widetilde{u})|\nabla \widetilde{v}|^2,\quad &in\quad S_1;\\
-div (\beta(\widetilde{u})\nabla \widetilde{v})=0, \quad &in\quad S_1,
\end{cases}
\end{align}
with the boundary data
\begin{eqnarray}
\begin{cases}
\widetilde{u}(x,t)=\phi(x_1+e^{-\frac{1}{2}}x),\quad &if\quad x_1+e^{-\frac{1}{2}}x\in \partial M;\\
\widetilde{v}(x,t)= \psi(x_1+e^{-\frac{1}{2}}x),\quad &if\quad x_1+e^{-\frac{1}{2}}x\in \partial M.
\end{cases}
\end{eqnarray}
Moreover, we have
\begin{align*}
\sup_{S_1}|\nabla \widetilde{u}|^2=e^{-1}\sup_{P^M_{e^{-1/2}}(z_1)}|\nabla u|^2 \leq e^{-1}\sup_{P^M_{\rho+e^{-1/2}}(z_0)}|\nabla u|^2
\leq e^{-1}\sup_{P^M_{\frac{1+\rho}{2}}(z_0)}|\nabla u|^2\leq 4
\end{align*}
and
\[
|\nabla \widetilde{u}|^2(0)=e^{-1}|\nabla u|^2(z_1)=1.
\]

Since $\widetilde{v}$ satisfies $$|\Delta \widetilde{v}|\leq C|\nabla \widetilde{u}||\nabla \widetilde{v}|,$$ by the standard elliptic estimates, for any $1<p<+\infty$ we have
\begin{equation}\label{inequality:14}
\sup_{-1\leq t\leq 0}\|\widetilde{v}\|_{W^{2,p}(D_{\frac{7}{8}}(0))}\leq C(p)(\|\nabla\widetilde{v}\|_{L^{p}(D_{1}(0))}+\|\psi\|_{C^2(M)}).
\end{equation}
Taking first $p=2$ and using \eqref{inequality:14} again ($p>2$), by Sobolev embedding for any $0<\beta=1-2/p<1$,
\begin{equation}
\sup_{-1\leq t\leq 0}\|\widetilde{v}\|_{C^{1+\beta}(D_{\frac{3}{4}}(0))}\leq C(\beta).
\end{equation}

Next, we want to show that there exists a constant $C>0$ such that
\begin{equation}\label{inequality:18}
1\leq C\int_{S_{3/4}}|\nabla \widetilde{u}|^2dxdt.
\end{equation}
If $C$ does not exist, we can find a sequence $\{(\widetilde{u}_i,\widetilde{v}_i)\}$ satisfying
\begin{align}\label{equation:01}
\begin{cases}
\partial_t \widetilde{u}_i=\Delta \widetilde{u}_i+ A(\widetilde{u}_i)(\nabla \widetilde{u}_i,\nabla \widetilde{u}_i)+(\nabla^N\beta)(\widetilde{u}_i)|\nabla \widetilde{v}_i|^2,\quad &in\quad S_1;\\
-div (\beta(\widetilde{u}_i)\nabla \widetilde{v}_i)=0, \quad &in\quad S_1,
\end{cases}
\end{align}
with the boundary data
\begin{eqnarray}\label{equation:02}
\begin{cases}
\widetilde{u}_i(x,t)=\phi(x_1+e^{-\frac{1}{2}}x),\quad &if\quad x_1+e^{-\frac{1}{2}}x\in \partial M;\\
\widetilde{v}_i(x,t)= \psi(x_1+e^{-\frac{1}{2}}x),\quad &if\quad x_1+e^{-\frac{1}{2}}x\in \partial M.
\end{cases}
\end{eqnarray}
and
\begin{equation}\label{inequality:07}
\sup_{S_{3/4}}\big(|\nabla \widetilde{u}_i|+|\nabla\widetilde{v}_i|\big)\leq C,
\end{equation}
\begin{equation}\label{inequality:17}
|\nabla \widetilde{u}_i|^2(0)=1,
\end{equation}
\begin{equation}\label{inequality:16}
\int_{S_{3/4}}|\nabla \widetilde{u}_i|^2dxdt\leq\frac{1}{i}.
\end{equation}

By a similar argument as in \textbf{Step 1} (since $(\widetilde{u_i},\widetilde{v_i})$ satisfy \eqref{equation:01}, \eqref{equation:02} and \eqref{inequality:07}), for any $0<\beta<1$, we have
\begin{align}
\|\nabla \widetilde{u}_i\|_{C^{\beta,\beta/2}(S_{1/2}(0))}\leq
C(\beta).
\end{align}
Therefore, there exist a subsequence of $\{\widetilde{u}_i\}$ (still denoted by $\{\widetilde{u}_i\}$) and a function $\overline{u}\in C^{1+\gamma,\gamma/2}(S_{1/2})$ such that
\[
\nabla \widetilde{u}_i\to \nabla \overline{u} \quad in \quad C^{\gamma,\gamma/2}(S_{1/2})
\]
where $0<\gamma<\beta$. Then by \eqref{inequality:16}, we know
\begin{equation}
\int_{S_{1/2}}|\nabla \overline{u}|^2dxdt=0
\end{equation}
which implies $\nabla \overline{u}\equiv 0$ in $S_{1/2}$. But, \eqref{inequality:17} tells us $|\nabla \overline{u}|(0)=1$. This is impossible and then \eqref{inequality:18} must be true.
Thus, we have
\begin{align*}
1\leq C\int_{S_{3/4}}|\nabla \widetilde{u}|^2dxdt &\leq
C\sup_{-1<t<0}\int_{B^M_{e^{\frac{1}{2}}}(x_1)}|\nabla u|^2(t_1+e^{-1}t)dx\\& \leq
C\sup_{-1<t<0}\int_{B^M_1(z_0)}|\nabla u|^2(t)dx\leq
C\epsilon_1.
\end{align*}
By choosing $\epsilon_1>0$ sufficiently small, it leads to a contradiction. Therefore we must have $(1-\rho)^2e\leq 4$ and then
\[
(1-3/4)^2\sup_{P^M_{3/4}(z_0)}|\nabla u|^2\leq (1-\rho)^2e\leq 4.
\]
Since $v$ satisfies $|\Delta v|\leq C(N)|\nabla u||\nabla v|$, $\|\nabla u\|_{L^\infty(P^M_{3/4}(z_0))}\leq 8$, $\|\nabla v\|_{L^4(M)}\leq C$ and $v|_{\partial M}=\psi\in C^{2+\alpha}(M)$, by the elliptic estimates for the  Laplace  operator and Sobolev embedding, we  easily get $$\|\nabla v\|_{L^\infty(P^M_{1/2}(z_0))}\leq C.$$ Thus we obtain the inequality \eqref{inequality:01} and finish the proof of the lemma.
\end{proof}

\

\section{short-time existence results}

To prove the local existence for the equations (\ref{heateq}), we first state some properties of the Dirichlet heat kernel when the dimension of the domain is  $2$. Let $G=G(x, y, t)$ be the heat kernel. We have

\begin{lemma}\label{estimateheatkernel} (estimates for Dirichlet heat kernel, see \cite{Ch}, \cite{Jost1})
For any $\alpha>0$, there exists a constant $c(\alpha)$ such that
$$G(x, y, t)\leq c(\alpha)t^{\alpha-1}dist(x, y)^{-2\alpha},$$
$$\|\nabla G(x, y, t)\|\leq c(\alpha)t^{\alpha-2}dist(x, y)^{1-2\alpha}.$$
\end{lemma}

By using the above lemma, we can give the proof of Theorem \ref{thm:shortime-existence}.

\

\begin{proof}[\textbf{Proof of Theorem \ref{thm:shortime-existence}}]
For $\epsilon>0$ and $u\in C^1(M\times [0,\epsilon])$ and
$v\in C^1(M)$, we define the space
\begin{eqnarray*}
X_\epsilon&=&\{u|_{M\times \{t=0\}}=u|_{\partial M\times [0,\epsilon]}=\phi, u \in C^1(M)\ \  {\text{for any fixed}}\ \  t\in [0, \epsilon]
, \|u\|_{X}<+\infty\},
\end{eqnarray*}
where the norm of $X_\epsilon$ is defined by
$$
\|u\|_{X}:=\|u\|_{C^0(M\times [0, \epsilon])}
+\sup_{t\in [0,\epsilon]}\|\nabla u(.,t)\|_{C^0(M)}.
$$

For a solution $(u,v)$ of \eqref{heateq}, we claim that
\begin{equation}\label{v}
\|v\|_{X}<+\infty, \ \ {\text{if}}\ \  \|u\|_{X}<+\infty.
\end{equation}
Notice that $v$ satisfies the following equation
\begin{equation}\label{heateqv}
\left\{
\begin{array}{cc}
&-\Delta v=2\frac{B^\top(u)}{\beta(u)}\nabla u\cdot\nabla v,\\
&v|_{\partial \Omega}=\psi.
\end{array}
\right.
\end{equation}
The elliptic estimates for \eqref{heateqv} tell us that, for any $p>1$, we have
\begin{eqnarray}\label{estimateforv4}
\|v\|_{W^{1,p}}\leq C\|v\|_{W^{2,2}}&\leq& C\left\|\frac{B^\top(u)}{\beta(u)}\nabla u\cdot\nabla v\right\|_2
+C\|\psi\|_{W^{2,2}}\nonumber\\
&\leq& C\|\nabla u\|_{C^0}\|\nabla v\|_2+C\|\psi\|_{W^{2,2}}.
\end{eqnarray}
Therefore, for any $p>1$, we have the $W^{2,p}$ estimate
\begin{eqnarray}\label{estimateforv5}
\|v\|_{W^{2,p}}&\leq& C\|\nabla u\|_{C^0}\|\nabla v\|_{p}+C\|\psi\|_{W^{2,p}}\nonumber\\
&\leq& C\|\nabla u\|_{C^0}\{\|\nabla u\|_{C^0}\|\nabla v\|_2+C\|\psi\|_{W^{2,2}}\}+C\|\psi\|_{W^{2,p}}.
\end{eqnarray}
By (\ref{estimateforv1}), it implies that
\begin{eqnarray*}
\|v\|_{W^{2,p}}&\leq& C(\psi)(\|u\|_X^2+\|u\|_X)+C(\psi).
\end{eqnarray*}
Therefore, by the Sobolev embedding theorem we have, for some $\alpha\in (0,1)$,
\begin{equation}\label{estimateforv9}
\|v\|_{C^{1,\alpha}}\leq C(\psi)+C(\psi)\|u\|_X+C(\psi)\|u\|_X^2
\end{equation}
This proves the claim.

Define
$$u_0(x,t)=\int_{M} G(x,y,t)\phi(y)dy-\int_0^t\int_{\partial M}\frac{\partial G}{\partial n}(x-y,t-s)\phi(y)d\sigma ds.$$
We consider the operator ${\mathbb{T}}:X_\epsilon\rightarrow X_\epsilon$
\begin{eqnarray*}
{\mathbb{T}} u(x,t)&=&u_0(x,t)-\int_0^t\int_{M}G(x-y,t-s)(A(\nabla u,\nabla u)-B^{\top}|\nabla v|^2)(y,s) dy d s.
\end{eqnarray*}
For $\delta>0$, we define
$$B_\delta:=\left\{u\in X_\epsilon, \|u-u_0\|_X\leq \delta\right\}.$$

To prove the existence of a local solution, we need\\
i): $ {\mathbb{T}}: B_\delta\rightarrow B_\delta;$ \\
ii): ${\mathbb{T}}$ is a contraction mapping in $B_\delta$.

{\it Proof of i)}. For $u\in B_\delta$, we have
$$
{\mathbb{T}} u-u_0=-\int_0^t\int_{M}G(x-y,t-s)
(A(\nabla u,\nabla u)-B^{\top}|\nabla v|^2)(y,s) dy d s.
$$
Notice that, for any $u\in B_\delta$,
$$\|u\|_X\leq \|u-u_0\|_X+\|u_0\|_X\leq C.$$
By (\ref{v}) we know that
$$\|v\|_X\leq C.$$

Letting $\alpha\in (1,\frac{3}{2})$ in Lemma \ref{estimateheatkernel}.
For any $(x,t)\in M\times [0, \epsilon)$, we have
\begin{eqnarray}\label{contraction1}
\left|{\mathbb{T}} u-u_0\right|(x,t)&\leq&
C\int_0^t\int_{M}G(x-y,t-s)
\left(|\nabla u|^2+|\nabla v|^2\right) dy d s\nonumber\\
&\leq&C\left\{\sup_{t\in [0,\epsilon]}\|\nabla u\|^2_{C^0(M)}+C\right\}\cdot\int_0^t\int_{\Omega}G(x-y,t-s) dy d s\nonumber\\
&\leq&C \int_0^\epsilon (\epsilon-s)^{\alpha-1}\int_{M}d(x, y)^{-2\alpha}dy\nonumber\\
&\leq& C\epsilon^{\alpha}.
\end{eqnarray}
Furthermore, we have
\begin{eqnarray}\label{contraction2}
\left|\nabla({\mathbb{T}} u-u_0)\right|(x,t)&\leq&
C\int_0^t\int_{M}|\nabla_x G|(x-y,t-s)
\left(|\nabla u|^2+|\nabla v|^2\right) dy d s\nonumber\\
&\leq&C\int_0^t\int_{M}|\nabla_x G|(x-y,t-s) dy d s\nonumber\\
&\leq&C \int_0^\epsilon (\epsilon-s)^{\alpha-2}\int_{M}d(x, y)^{1-2\alpha}dy\nonumber\\
&\leq& C\epsilon^{\alpha-1},
\end{eqnarray}
where $C$ is a  constant which depends on the norm $\|u\|_X$ .
Then (\ref{contraction1}) and (\ref{contraction2}) give us that, for any $\delta>0$, there exists $\epsilon>0$, such that
${\mathbb{T}}$ is a map from $B_\delta$ in to $B_\delta$.

{\it Proof of ii).} We need to show that there exists $\rho\in (0,1)$ such that, for any $u_1,u_2\in B_\delta$,
$$
\|{\mathbb{T}} u_1-{\mathbb{T}} u_2\|_X\leq \rho \|u_1-u_2\|_X.
$$

We have
\begin{eqnarray}\label{contraction3}
\left|{\mathbb{T}} u_1-{\mathbb{T}} u_2\right|(x,t)&&\leq
\int_0^t\int_{M}G(x-y,t-s)
|A(u_1)(\nabla u_1,\nabla u_1)-A(u_2)(\nabla u_2,\nabla u_2)\nonumber\\
&&+B^{\top}(u_1)|\nabla v_{1}|^2-B^{\top}(u_2)|\nabla v_{2}|^2| dy d s.
\end{eqnarray}
Firstly, we estimate
\begin{eqnarray*}
&&|A(u_1)(\nabla u_1,\nabla u_1)-A(u_2)(\nabla u_2,\nabla u_2)|\\
&\leq&|A(u_1)(\nabla u_1,\nabla u_1)-A(u_1)(\nabla u_2,\nabla u_2)|
+|A(u_1)(\nabla u_2,\nabla u_2)-A(u_2)(\nabla u_2,\nabla u_2)|\\
&\leq& C\{(|\nabla u_1|+|\nabla u_2|)|\nabla(u_1-u_2)|+|\nabla u_2|^2|u_1-u_2|\}\\
&\leq& C\|u_1-u_2\|_X.
\end{eqnarray*}

In the following we estimate $|B^{\top}(u_1)|\nabla v_{1}|^2-B^{\top}(u_2)|\nabla v_{2}|^2|$. From \eqref{heateqv}, we have
$$div(\beta(u_1)\nabla v_{1})-div(\beta(u_2)\nabla v_{2})=0.$$
Multiply by $(v_{1}-v_{2})$ and integrate on $M$. We have
\begin{eqnarray}\label{estimateforv6}
0&=&\int_{M}div (\beta(u_1)\nabla v_{1}-\beta(u_2)\nabla v_{2})(v_{1}-v_{2})dx\nonumber\\
&=&\int_{M}\langle \beta(u_1)(\nabla v_{1}-\nabla v_{2})
+(\beta(u_1)-\beta(u_2))\nabla v_{2}, \nabla (v_{1}-v_{2}) \rangle dx,
\end{eqnarray}
which implies that
\begin{eqnarray}\label{estimateforv7}
\lambda\int_{M} |\nabla v_{1}-\nabla v_{2}|^2dx&\leq&\int_{M} \beta(u_1)|\nabla v_{1}-\nabla v_{2}|^2dx\nonumber\\
&=&\int_{M}\langle (\beta(u_1)-\beta(u_2))\nabla v_{2}, \nabla (v_{1}-v_{2})\rangle dx\nonumber\\
&\leq& \frac{\lambda}{2}\int_{M} |\nabla v_{1}-\nabla v_{2}|^2dx
+\frac{1}{2\lambda}\int_{M}|\beta(u_1)-\beta(u_2)|^2|\nabla v_{2}|^2dx.
\end{eqnarray}
From \eqref{estimateforv7}, we get
\begin{eqnarray}\label{estimateforv8}
\int_{M} |\nabla v_{1}-\nabla v_{2}|^2dx&\leq&C\int_{M}|\beta(u_1)-\beta(u_2)|^2|\nabla v_{2}|^2dx\nonumber\\
&\leq&C\|u_1-u_2\|_{C^0}^2\int_{M} |\nabla v_{2}|^2 dx\nonumber\\
&\leq&C(\psi)\|u_1-u_2\|_X^2.
\end{eqnarray}
By \eqref{heateqv}, we have
\begin{equation}\label{equationforv1v2}
\Delta (v_{1}-v_{2})=\gamma(u_1)\nabla u_1\cdot(\nabla v_{1}-\nabla v_{2})
+2\left(\gamma(u_1)\nabla u_1-\gamma(u_2)\nabla u_2\right)
\cdot\nabla v_{2}
\end{equation}
where $\gamma(u):=2\frac{B^\top(u)}{\beta(u)}$, with the boundary condition
$$(v_{1}-v_{2})|_{\partial M}=0.$$
The elliptic estimates for \eqref{equationforv1v2} tell us that
\begin{eqnarray}\label{estimateforv1v21}
\|v_{1}-v_{2}\|_{W^{2,2}}&\leq& C\|u_1\|_X\|\nabla v_{1}-\nabla v_{2}\|_2\nonumber\\
&&+C\|\nabla (u_1-u_2)\|_{C^0}\|\nabla v_{2}\|_2+C\|u_1-u_2\|_{C^0}\|\nabla u_2\|_{C^0}\|\nabla v_{2}\|_2.
\end{eqnarray}
By \eqref{estimateforv8}, we have, for any $p>1$,
\begin{eqnarray}\label{estimateforv1v22}
\|v_{1}-v_{2}\|_{W^{1,p}}
\leq C\|v_{1}-v_{2}\|_{W^{2,2}}
\leq C\|u_1-u_2\|_X
\end{eqnarray}
Applying this estimate and \eqref{estimateforv4} to \eqref{equationforv1v2}, we have
\begin{eqnarray}\label{estimateforv1v23}
\|v_{1}-v_{2}\|_{W^{2,p}}
&\leq& C\|u_1\|_X\|\nabla v_{1}-\nabla v_{2}\|_p\nonumber\\
&&+C\|\nabla (u_1-u_2)\|_{C^0}\|\nabla v_{2}\|_p+C\|u_1-u_2\|_{C^0}\|\nabla u_2\|_{C^0}\|\nabla v_{2}\|_p\nonumber\\
&\leq& C(\psi)\|u_1-u_2\|_X.
\end{eqnarray}
\eqref{estimateforv1v23} implies that
\begin{equation}\label{estimateforv1v24}
\|v_{1}-v_{2}\|_{C^{1+\alpha}}\leq C\|u_1-u_2\|_X.
\end{equation}
Therefore, we get from \eqref{estimateforv9} and \eqref{estimateforv1v24},
\begin{eqnarray}\label{contraction4}
&&|B^{\top}(u_1)|\nabla v_{1}|^2-B^{\top}(u_2)|\nabla v_{2}|^2|\nonumber\\
&&\leq (B^{\top}(u_1)-B^{\top}(u_2))|\nabla v_{1}|^2+B^{\top}(u_2)\langle\nabla v_{1},\nabla (v_{1}-v_{2})\rangle
\nonumber\\
&&\quad+B^{\top}(u_2)\langle\nabla v_{2},\nabla (v_{1}-v_{2})\rangle\nonumber\\
&&\leq C |\nabla v_{1}|^2|u_1-u_2|+C|\nabla v_{1}||\nabla (v_{1}-v_{2})|
+C|\nabla v_{2}||\nabla (v_{1}-v_{2})|\nonumber\\
&&\leq C(\psi)\|u_1-u_2\|_X.
\end{eqnarray}
\eqref{contraction3} and \eqref{contraction4} give us
\begin{equation}\label{contraction5}
\left|{\mathbb{T}} u_1-{\mathbb{T}} u_2\right|(x,t)\leq C(\psi)\epsilon\|u_1-u_2\|_X.
\end{equation}
Similarly, we can also show that
\begin{equation}\label{contraction6}
\left|\nabla ({\mathbb{T}} u_1-{\mathbb{T}} u_2)\right|(x,t)\leq C(\psi)\epsilon\|u_1-u_2\|_X.
\end{equation}
Then we can conclude the claim that ${\mathbb{T}}$ is a contraction mapping in $B_\delta$ which implies immediately that there exists a unique fixed point $(u,v)\in B_\delta$ of ${\mathbb{T}}$ such that $(u,v)$ solves (\ref{heateq}).

{\it Regularity:} Since $\nabla u,\nabla v\in L^\infty([0,\epsilon]\times M)$, according to the classical $W^{2,1}_p$ estimates of second order parabolic equations, for any $p>1$, we have $$\|u\|_{W^{2,1}_p([0,\epsilon]\times M)}\leq C(1+\|\phi\|_{W^{2,p}}(M))$$ which implies $\nabla u\in C^{\alpha,\frac{\alpha}{2}}([0,\epsilon]\times M)$ by Sobolev embedding.
The Schauder estimates for parabolic and elliptic equations give us that the fixed point $(u,v)$ has the desired regularity. (See Lemma \ref{lem:small-energy-regularity} for a similar argument.)

We still need to show that if the image of the initial map $\phi(M)\subset N$, along the flow, we have $u(M \times [0,\epsilon))\subset N$. To this end, let $\pi: M_\sigma\rightarrow M$ be the smooth nearest point projection map. We now compute the evolution equation of
$$
\rho(u):=|\pi(u)-u|^2.
$$
\begin{eqnarray}\label{heatrho}
\frac{1}{2}(\partial_t-\Delta) \rho&=&\langle\pi -u, \partial_t( \pi-u)\rangle
-\left|\nabla(\pi-u)\right|^2-\langle\pi-u, \Delta(\pi-u)\rangle\nonumber\\
&=&\langle\pi -u, d\pi(\partial_t u)-\partial_t u\rangle
-\left|\nabla(\pi-u)\right|^2
-\left\langle\pi-u, \nabla\cdot(d\pi(\nabla u))-\Delta u\right\rangle\nonumber\\
&=&-\left|\nabla(\pi-u)\right|^2
+\left\langle\pi-u, d\pi((\partial_t-\Delta)u)\right\rangle\nonumber\\
&&+\left\langle\pi-u, A(\nabla u,\nabla u)-(\partial_t-\Delta)u\right\rangle\nonumber\\
&=&-\left|\nabla(\pi-u)\right|^2
+\left\langle\pi-u, d\pi((\partial_t-\Delta)u)\right\rangle+\left\langle\pi-u, B^{\top}|\nabla v|^2\right\rangle\nonumber\\
&=&-\left|\nabla(\pi-u)\right|^2\leq 0.
\end{eqnarray}
To get the last equality of (\ref{heatrho}), we shall notice that, for $p\in N$, $d\pi(p)$ is an orthogonal projection to $T_p N$, $B^{\top}$ is orthogonal to $(\pi-u)$. Since $\rho(u)(\cdot, 0)=\rho(u)|_{\partial M\times [0,\epsilon]}=0$, by the maximum principle, we have $\rho(u)\equiv 0$.

{\it Finite singularities:} By Lemma \ref{lem:small-energy-regularity}, we know that the maximum existence time $T_1$ is characterized by $$\limsup_{x\in M,t\to T_1}E(u;B^M_r(x))>\epsilon_1\mbox{ for all } r>0.$$
We just need to prove that the singular set $S(u,T_1)$ at the singular time $T_1$ is a finite set.

Let $\{x_j\}_{j=1}^J$ be any finite subset of $S(u,T_1)$. Then we have $$\limsup_{t\to T_1}E(u;B^M_r(x_j))>\epsilon_1\mbox{ for all } r>0,\ 1\leq j\leq J.$$ Therefore, we can choose $R>0$ such that $\{B^M_{2R}(x_j)\}_{j=1}^J$ are mutually disjoint. By Lemma \ref{lem:two-balls}, we get that for any $1\leq j\leq J$ and any $s\in [T_1-\frac{\epsilon_1R^2}{2C},T_1)$, there holds
\begin{align*}
J\epsilon_1\leq \sum_{j=1}^J\limsup_{t\to T_1}E(u(t);B^M_R(x_j))&\leq \sum_{j=1}^J\left(E(u(s);B^M_{2R}(x_j))+\frac{\epsilon_1}{2}\right)\\
&\leq E(u(s))+\frac{J\epsilon_1}{2}
\end{align*}
which implies $$J\leq \frac{2(1+\Lambda)E(\phi,\psi)}{\epsilon_1}.$$ So, we proved the finiteness of $S(u,T_1)$.

{\it Uniquness:} Suppose $(u_1, v_1), (u_2, v_2)$ are two solutions of  \eqref{heateq}  and \eqref{boundary-data}. Let $\tilde{u}=u_1-u_2,$  $\tilde{v}=v_1-v_2$. By \eqref{contraction5} and \eqref{contraction6}, we know that, for any $\varepsilon>0$,
\begin{eqnarray*}
\|\tilde{u}\|_X\leq \varepsilon \|\tilde{u}\|_X.
\end{eqnarray*}
It implies that $\tilde u\equiv 0$. Then by \eqref{estimateforv1v24}, we get that $\tilde v\equiv 0$.
\end{proof}

\

\section{behavior of singularities}

In this section, we use the blow up analysis to study the behavior of singularities at the singular time of the solution derived by Theorem \ref{thm:shortime-existence}. We will prove Theorem \ref{thm:blow-up} in this section.

First, we recall a removable singularity theorem which will be used in this section.
\begin{theorem}[Theorem 3.4 in \cite{hanzhaozhu}]\label{thm:removable-singu}
Let $(u,v)$ be a smooth harmonic map from the punctured disk $D\setminus \{0\}$ to $(N\times \R,g-\beta (d\theta)^2)$ with bounded energy $E(u,v;M)<\infty$, where $D\subset\R^2$ is the unit disk, then $(u,v)$ extends to the whole disk $D$.\end{theorem}
\begin{proof}
We repeat the idea of the proof of \cite{hanzhaozhu} here for completeness. By a similar argument as in Lemma A.2 in \cite{Jost}, it is easy to see that $(u,v)$ is a weakly harmonic map from $D$ into $N\times\R$. Then the regularity Theorem 1.3 in \cite{zhu} gives that $(u,v)$ is smooth in $D$ and hence the singularity point $\{0\}$ is removable.
\end{proof}

\

\begin{proof}[\textbf{Proof of Theorem \ref{thm:blow-up}}]
Let $(x_0,T_1)$ be a singularity. Without loss of generality, we assume $0<T_1<\infty$. The proof of $T_1=\infty$ is similarly.

Since there are at most finitely many singular points at the singular time $T_1$, we may assume $$\nabla u\in C_{loc}^{\alpha,\alpha/2}(B^M_\delta(x_0)\times [T_1-\delta^2,T_1]\setminus\{(x_0,T_1)\})$$ for some $\delta>0$. By Lemma \ref{lem:small-energy-regularity}, there exist sequences $t_i\nearrow T_1$, $x_i\to x_0$, $r_i\to 0$ such that
\begin{align}\label{equation:blowup-rules}
E(u(\cdot,t_i),B^M_{r_i}(x_i))=\sup_{\substack{(x,t)\in B^M_\delta(x_0)\times [T_1-\delta^2,t_i]\\B^M_r(x)\subset B^M_\delta(x_0),r\leq r_i}}E(u(\cdot,t_i),B^M_{r}(x))=\frac{\epsilon_1}{2}.
\end{align}
According to Lemma \ref{lem:two-balls}, for any $T_1-\delta^2\leq s\leq t_i<T_1$, there holds
\begin{align*}
E(u(t_i);B^M_{r_i}(x_i))\leq E(u(s);B^M_{2r_i}(x_i))+C_1\frac{t_i-s}{r_i^2}+C_2(t_i-s),
\end{align*}
where $C_1$ and $C_2$ are the constants in Lemma \ref{lem:two-balls}. Setting $\tau=\frac{\epsilon_1}{8C_1}+\frac{\epsilon_1}{8C_2}$, then for any $s\in [t_i-\tau r_i^2,t_i]$, we get
\begin{align}\label{inequality:08}
E(u(s);B^M_{2r_i}(x_i))\geq\frac{\epsilon_1}{4}.
\end{align}

We first deal with the second statement in the theorem.

\

\textbf{Step 1:} Let $x_0\in \partial M$ and we show the statement $(1)$ holds under the assumption $$\limsup_{i\to \infty}\frac{dist(x_i,\partial M)}{r_i}=\infty.$$ After passing to a subsequence, we may assume $\lim_{i\to \infty}\frac{dist(x_i,\partial M)}{r_i}=\infty$. As $i$ tends to infinity, we can assume $t_i-\frac{\delta^2}{4}>T_1-\delta^2$. Denote $$D_i:=\{x\in\R^2|x_i+r_ix\in B^M_\delta(x_0)\}$$ and $$u_i(x,t):=u(x_i+r_ix,t_i+r_i^2t); v_i(x,t):=v(x_i+r_ix,t_i+r_i^2t).$$ It is easy to see that $(u_i,v_i)$ lives in $D_i\times [-\frac{\delta^2}{4r_i^2},0]$ which tends to $\R^2\times (-\infty,0]$ as $i\to \infty$ and satisfies
\begin{align}
\begin{cases}
\partial_t u_i=\Delta u_i+ A(u_i)(\nabla u_i,\nabla u_i)+B^\top(u_i)|\nabla v_i|^2,\ &in\ D_i\times [-\frac{\delta^2}{4r_i^2},0] \\
-div (\beta(u_i)\nabla v_i)=0, \ &in\ D_i\times [-\frac{\delta^2}{4r_i^2},0]
\end{cases}
\end{align}
with the boundary data
\begin{align}
\begin{cases}
u_i(x,t)=\phi(x_i+r_ix), \  &if\  x_i+r_ix\in\partial M,\\
v_i(x,t)=\psi(x_i+r_ix),\  &if \  x_i+r_ix\in\partial M.
\end{cases}
\end{align}
By Lemma \ref{lem:Dirichlet-energy} and Corollary \ref{cor:energy-t}, for any $\tau$ we have
\begin{align}\label{inequality:09}
\int_{-\tau}^0\int_{D_i}|\partial_tu_i|^2dxdt\leq\int_{t_i-r_i^2\tau}^{t_i}\int_M|\partial_t u|^2dMdt  \to 0, \quad {\text as} \ i\to \infty,
\end{align}
and
\begin{equation}\label{inequality:10}
\sup_{-\frac{\delta^2}{4r_i^2}\leq t\leq0}\|dv_i\|_{L^2(D_i)}\leq\sup_{T_1-\delta^2\leq t\leq T_1}\|dv\|_{L^2(M)}\leq C,
\end{equation}
\begin{equation}\label{inequality:11}
\sup_{-\frac{\delta^2}{4r_i^2}\leq t\leq0}\|du_i\|_{L^2(D_i)}\leq\sup_{T_1-\delta^2\leq t\leq T_1}\|du\|_{L^2(M)}\leq C.
\end{equation}
By \eqref{equation:blowup-rules}, we can see that
\begin{align*}
\sup_{-\tau\leq t\leq 0}\sup_{x\in D_i}\int_{B_1(x)\cap D_i}|\nabla u_i|^2(y,t)dy\leq
\sup_{\substack{(x,t)\in B^M_\delta(x_0)\times [T_1-\delta^2,t_i]\\B^M_r(x)\subset B^M_\delta(x_0),\ r\leq r_i}}E(\Phi(t),B^M_{r}(x))=\frac{\epsilon_1}{2}.
\end{align*}
So, for any $x\in \mathbb{R}^2$, when $i$ is sufficiently large, we have
\begin{equation}\label{inequality:12}
\sup_{-\tau\leq t\leq 0}\int_{B_1(x)}|\nabla u_i|^2(y,t)dy\leq\frac{\epsilon_1}{2}.
\end{equation}
Combining \eqref{inequality:10}, \eqref{inequality:12} with Lemma \ref{lem:small-energy-regularity}, we have
\begin{equation}\label{equation:03}
\sup_{-\frac{\tau}{2}\leq t\leq 0}\|v_i(\cdot,t)\|_{C^{2+\alpha}(B_{1/2}(x))}+\sup_{-\frac{\tau}{2}\leq t\leq 0}\|u_i(\cdot,t)\|_{C^{1+\alpha}(B_{1/2}(x))}\leq C,
\end{equation}
which tells us
\begin{equation}\label{equation:04}
\sup_{-\frac{\tau}{2}\leq t\leq 0}\|v_i(\cdot,t)\|_{C_{loc}^{2+\alpha}(\mathbb{R}^2)}
+\sup_{-\frac{\tau}{2}\leq t\leq 0}\|u_i(\cdot,t)\|_{C_{loc}^{1+\alpha}(\mathbb{R}^2)}\leq C.
\end{equation}
From \eqref{inequality:09} and \eqref{equation:04}, we can find $\sigma_i\in[-\frac{\tau}{2},0]$ such that as $i\to\infty$, there holds
\begin{align}
\int_{D_i}|\partial_tu_i|^2(x,\sigma_i)dx \to 0
\end{align}
and
\begin{equation}
\|v_i(\cdot,\sigma_i)\|_{C_{loc}^{2+\alpha}(\mathbb{R}^2)}
+\|u_i(\cdot,\sigma_i)\|_{C_{loc}^{1+\alpha}(\mathbb{R}^2)}\leq C.
\end{equation}
Therefore, there exist $\widetilde{u}\in C^1_{loc}(\mathbb{R}^2)$, $\widetilde{v}\in C^2_{loc}(\mathbb{R}^2)$ and a subsequence of $(u_i(\cdot,\sigma_i),v_i(\cdot,\sigma_i))$ such that
\begin{align*}
u_i(\cdot,\sigma_i)\to\widetilde{u}\quad in \quad C^1_{loc}(\mathbb{R}^2)\mbox{ and }v_i(\cdot,\sigma_i)\to\widetilde{v}\quad in \quad C^2_{loc}(\mathbb{R}^2).
\end{align*}
Setting $t=\sigma_i$ in the equation \eqref{equation:01} and letting $i\to\infty$, it is easy to see that $(\widetilde{u},\widetilde{v})$ is a harmonic map from $\R^2$ into the Lorentzian manifold $N\times \R$ with
\[
\frac{\epsilon_1}{4}\leq\|\nabla\widetilde{u}\|_{L^2(\mathbb{R}^2)}\leq C,\ \|\nabla\widetilde{v}\|_{L^2(\mathbb{R}^2)}\leq C.
\]
Here we use \eqref{inequality:08} and \eqref{inequality:11}. Let $f:\R^2\to {\mathbb{S}}^2\setminus\{S\}$ be the stereographic projection, where $S$ is the south pole of the sphere. Due to the conformal invariance and removable singularity Theorem \ref{thm:removable-singu}, $(\widetilde u(f^{-1}(x)),\widetilde v(f^{-1}(x)))$ is a harmonic map from ${\mathbb{S}}^2$ into the Lorentzian manifold $N\times \R$. For simplicity, we still denote $(\widetilde u(f^{-1}(x)),\widetilde v(f^{-1}(x)))$ by $(\widetilde u,\widetilde v)$. It is clear that $\widetilde v$ satisfies the equation $div(\beta(\widetilde u)\nabla \widetilde v)=0$ in ${\mathbb{S}}^2$ with finite energy $\|\nabla\widetilde v\|_{L^2({\mathbb{S}}^2)}\leq C$. It follows that $\widetilde v$ must be a constant map. Then $\widetilde u$ is a nontrivial harmonic sphere.

\

\noindent\textbf{Step 2}: If $x_0\in\partial M$, then $\limsup_{i\to\infty}\frac{dist(x_i,\partial M)}{r_i}\to \infty$.

\

If not, up to subsequence, we may assume $\frac{dist(x_i,\partial M)}{r_i}\to a$ as $i\to\infty$. Then
\[
D_i\to \mathbb{R}^2_a:=\{(x^1,x^2)|x^2\geq -a\}.
\]
Furthermore, we have that, for any $x\in\{x^2=-a\}$ on the boundary, $x_i+r_ix\to x_0$ and
\begin{align*}
u_i(x,t)=\varphi(x_i+r_ix)\quad {\text if} \quad x_i+r_ix\in\partial M;\\
v_i(x,t)=\psi(x_i+r_ix)\quad {\text if} \quad x_i+r_ix\in\partial M;
\end{align*}
By Lemma \ref{lem:small-energy-regularity} and \eqref{equation:blowup-rules}, for any $B_{4R}(0)\subset \mathbb{R}^2,R>0$, we have
\begin{equation}\label{equation:05}
\sup_{-\frac{\tau}{2}\leq t\leq 0}\|v_i(\cdot,t)\|_{C^{2+\alpha}(B_{4R}(0)\cap D_i)}
+\sup_{-\frac{\tau}{2}\leq t\leq 0}\|u_i(\cdot,t)\|_{C^{1+\alpha}(B_{4R}(0)\cap D_i)}\leq C.
\end{equation}
From \eqref{inequality:09} and \eqref{equation:05}, we can find $\sigma_i\in[-\frac{\tau}{2},0]$ such that as $i\to\infty$, we have
\begin{align}
\int_{D_i}|\partial_tu_i|^2(x,\sigma_i)dx \to 0
\end{align}
and
\begin{equation}\label{equation:06}
\|v_i(\cdot,\sigma_i)\|_{C^{2+\alpha}(B_{4R}(0)\cap D_i)}
+\|u_i(\cdot,\sigma_i)\|_{C^{1+\alpha}(B_{4R}(0)\cap D_i)}\leq C.
\end{equation}
Setting $d_i:=dist(x_i,\partial M)$ and $B_R^+(0):=\{(x_1,x_2)\in B_R(0)|x_2\geq 0\}$, for $i,R$ sufficiently large, \eqref{equation:06} implies
\begin{equation}
\|v_i(x-(0,\frac{d_i}{r_i}),\sigma_i)\|_{C^{2+\alpha}(B^+_{3R}(0))}
+\|u_i(x-(0,\frac{d_i}{r_i}),\sigma_i)\|_{C^{1+\alpha}(B^+_{3R}(0))}\leq C.
\end{equation}
Then there exist a subsequence of $(u_i,v_i)$ and a harmonic map $(\widetilde{u}^1,\widetilde{v}^1)\in C_{loc}^{2+\alpha}(\R^2_a,N\times\R)$ satisfying $\widetilde{u}^1|_{\partial^0\R_0^2}=\phi(x_0)$, $\widetilde{v}^1|_{\partial^0\R_0^2}=\psi(x_0)$ such that
\begin{align*}
\lim_{i\to\infty}\|v_i(x-(0,\frac{d_i}{r_i}),\sigma_i)-\widetilde{v}^1\|_{C^{1}(B^+_{3R}(0))}
&=0,\\
\lim_{i\to\infty}\|u_i(x-(0,\frac{d_i}{r_i}),\sigma_i)-\widetilde{u}^1\|_{C^{1}(B^+_{3R}(0))}
&=0.
\end{align*}
Set $\widetilde{u}(x):=\widetilde{u}^1(x+(0,a))$ and $\widetilde{v}(x):=\widetilde{v}^1(x+(0,a))$. We have
\begin{align*}
\lim_{i\to\infty}\|v_i(\cdot,\sigma_i)-\widetilde{v}\|_{W^{1,2}(B_{2R}(0)\cap D_i\cap \R_a^2)}
&=0,\\
\lim_{i\to\infty}\|u_i(\cdot,\sigma_i)-\widetilde{u}\|_{W^{1,2}(B_{2R}(0)\cap D_i\cap \R_a^2)}
&=0.
\end{align*}
Combining these with \eqref{equation:06} and noticing that the measure of $B_{2R}(0)\cap D_i\setminus \R_a^2$ and $B_{2R}(0)\cap\R_a^2\setminus D_i$ goes to zero, we have
\begin{align*}
\lim_{i\to\infty}\|v_i(\cdot,\sigma_i)\|_{W^{1,2}(B_{R}(0)\cap D_i)}=\|\widetilde{v}\|_{W^{1,2}(B_{R}(0)\cap \R_a^2 )},\\
\lim_{i\to\infty}\|u_i(\cdot,\sigma_i)\|_{W^{1,2}(B_{R}(0)\cap D_i)}=\|\widetilde{u}\|_{W^{1,2}(B_{R}(0)\cap \R_a^2 )}.
\end{align*}
According to \eqref{inequality:08},\eqref{inequality:10} and \eqref{inequality:11}, we obtain
\begin{align}\label{inequality:15}
\frac{\epsilon_1}{4}\leq\|\nabla\widetilde{u}\|_{L^2(\mathbb{R}_a^2)}\leq C;\ \|\nabla\widetilde{v}\|_{L^2(\mathbb{R}_a^2)}\leq C.
\end{align}

Due to the conformal invariance, we can take $(\widetilde{u},\widetilde{v})$ as a harmonic map from the unit disk $B_1(0)$ into the Lorentzian manifold $N\times \R$. Since $\widetilde{v}$ satisfies $$div(\beta(\widetilde{u})\nabla \widetilde{v})=0\ {\text in}\ B_1(0)$$ and $\widetilde{v}|_{\partial B_1(0)}\equiv\psi(x_0)$, $\widetilde{v}$ must be a constant map. Thus, $\widetilde{u}$ is a harmonic maps from $B_1(0)$
with constant boundary data $\widetilde{u}|_{\partial B_1(0)}=\phi(x_0)$ which should be a constant map \cite{Lemaire}. This is a contradiction with \eqref{inequality:15} and the second statement $(2)$ is proved.

For the first statement in the theorem, the argument is almost the same as what we have done in \textbf{Step 1} and we omit it here for brevity.
\end{proof}

\

\section{long time existence and convergence results}

In this section, we use  arguments from blow up analysis to prove some long time existence and convergence results for \eqref{heateq} and \eqref{boundary-data}. Theorem \ref{thm:small-initial-energy} and Theorem \ref{thm:convex-func-target-mfd} will be proved in this section.

\

\begin{proof}[\textbf{Proof of Theorem \ref{thm:small-initial-energy}}]
By the short time existence theorem, we just need to show that the solution $(u,v)$ does not blow up at any time $t\in(0,\infty]$.

If not, we may assume that $T_1$ is the first singular (or blow up) time and $(x_0,T_1)\in S(u,T_1)$ is a singular point (or energy concentration point), $i.e.$
$$\limsup_{t\to T_1}E(u;B^M_R(x_0))>\epsilon_1\mbox{ for all } R>0.$$ Let $t_i,x_i,r_i,\widetilde{u}$ be as in Theorem \ref{thm:blow-up}. Denote $$u_i(x):=u(x,t_i).$$
By Theorem \ref{thm:shortime-existence}, we know that the singular set $S(u,T_1)$ at the singular time $T_1$ is a finite set and we may denote it by $S(u,T_1):=\{p_1,...,p_J\}$. By Lemma \ref{lem:Dirichlet-energy}, we have $E(u_i;M)\leq E(\phi)$. Thus, there exists a weak limit in $W^{1,2}(M,N)$ which is denoted by $u(x,T_1)$ such that $u(x,T_1)|_{\partial M}=\phi$ and $$u_i(x)\rightharpoonup u(x,T_1)$$ as $i\to\infty$. Moreover, by the definition of $S(u,T_1)$ and Lemma \ref{lem:small-energy-regularity}, we know $$u_i(x)\to u(x,T_1)\ in\ C^2_{loc}(M\setminus S(u,T_1))$$ as $i\to\infty$. Therefore we have
\begin{align*}
\lim_{i\to\infty}E(u_i;M)&=\lim_{\delta\to 0}\lim_{i\to\infty}(E(u_i;M\setminus \cup_{j=1}^JB^M_\delta(p_j))+E(u_i;\cup_{j=1}^JB^M_\delta(p_j)))\\
&=E(u(x,T_1);M)+\sum_{j=1}^J\lim_{\delta\to 0}\lim_{i\to\infty}E(u_i;B^M_\delta(p_j))\\
&\geq E(u(x,T_1);M)+\lim_{\delta\to 0}\lim_{i\to\infty}E(u_i;B^M_\delta(x_0))\\
&\geq E(u(x,T_1);M)+\lim_{R\to\infty}\lim_{\delta\to 0}\lim_{i\to\infty}E(u_i;B^M_{r_iR}(x_i))\\
&= E(u(x,T_1);M)+E(\widetilde{u};S^2),
\end{align*}
where $\widetilde{u}$ is a nontrivial harmonic sphere.

By the definition of $\overline{\epsilon}_1,\overline{\epsilon}_2,\overline{\epsilon}$(see Theorem \ref{thm:small-initial-energy}) and Lemma \ref{lem:Dirichlet-energy}, we have
\begin{align*}
\overline{\epsilon}&\leq E(u(x,T_1);M)+E(\widetilde{u};S^2)\\&
\leq \lim_{i\to\infty}E(u_i;M)
\leq E(\phi)+(\Lambda-\lambda)E(\psi)<\overline{\epsilon},
\end{align*}
which is a contradiction.

By Corollary \ref{cor:energy-t}, we have
$$\int_0^\infty\int_M|\partial_t u|^2dxdt\leq C.$$
Then there exists a time sequence $t_i\to\infty$ such that $\partial_tu(x,t_i)\to 0,\ a.e.\ x\in M$. Since the flow dose not blow up, by Lemma \ref{lem:small-energy-regularity}, there exists a subsequence of $(u(x,t_i),v(x,t_i))$ which is still denoted by $(u(x,t_i),v(x,t_i))$ such that
$$u(x,t_i)\to u_\infty(x)\ {\text and}\ v(x,t_i)\to v_\infty(x)\ {\text in} \ C^2(M)$$
as $i\to\infty$, where $(u_\infty,v_\infty)$ is a harmonic map from $M$ to the Lorentzian manifold $N\times \R$ with the boundary data $(u_\infty,v_\infty)|_{\partial M}=(\phi,\psi)$. Then the theorem is proved.
\end{proof}

\

Now, we proceed to prove our last Theorem \ref{thm:convex-func-target-mfd}.

\begin{proof}[\textbf{Proof of Theorem \ref{thm:convex-func-target-mfd}}]
By the proof of Theorem \ref{thm:small-initial-energy}, we only need to prove that the solution $(u,v)$ does not blow up at any time $t\in(0,\infty]$.

If not, then we may assume $T_1$ is the first singular (or blow up) time and $(x_0,T_1)\in S(u,T_1)$ is a singular point (or energy concentration point). By Theorem \ref{thm:blow-up}, we can get a nontrivial harmonic sphere $\widetilde{u}:{\mathbb{S}}^2\to N$. However, under the assumptions of Theorem \ref{thm:convex-func-target-mfd}, the main result in \cite{lizhu} tells us that the manifold $N$ cannot admit any nontrivial harmonic sphere. This is a contradiction and we have finished  the proof.
\end{proof}

\section{An elliptic scheme}
In \cite{H-K-L} Hardt-Kinderlehrer-Lin studied the   existence and partial regularity of static liquid crystal and developed a method of general interest that can also be adapted to our problem. They considered the  functional
\begin{equation}\label{staticliquidcrystal}
\varepsilon^*(u,\psi)=\int_\Omega [W(\nabla u, u)-A(\nabla\psi, u)]dx,
\end{equation}
where $\Omega\subset \R^3$ is a smooth bounded domain, $(u,\psi)\in H^1(\Omega, S^2)\times H^1(\Omega)$, $W$ is the Ossen-Frank free  energy density and $A$ satisfies the coercivity condition $$A(\xi,u)\geq \lambda |\xi|^2$$ for $\xi\in\R^3$.
This functional is not bounded from below. Nevertheless, \cite{H-K-L} could apply a minimization scheme by imposing Gauss's law as a constraint and then obtain critical points. Let us outline their approach. For any $u\in H^1(\Omega,S^2)$, using the Euler-Lagrange equation of $\psi$ and the boundary condition that $\psi|_{\partial\Omega}=\varphi_0$, we know that there exists a unique solution denoted by $\Phi(u)$ which is the unique minimizer of $\int_\Omega A(\nabla\psi, u)dx$ among $\psi\in H^1(\Omega)$ with $\psi|_{\partial \Omega}=\varphi_0$. Moreover, there holds $$\|\nabla\Phi(u)\|_{L^2(\Omega)}\leq C(\lambda, \Omega,\varphi_0).$$ Then for $u\in H^1(\Omega, S^2)$, they let $\varepsilon(u)$ (or $\varepsilon_{\varphi_0}(u)$ to indicate the dependence on $\varphi_0$) denote the energy $$\varepsilon(u)=\varepsilon^*(u,\Phi(u)).$$ The functional $\varepsilon(u)$ only depends on $u$ and is bounded from below. Thus, the minimizer exists.

The structure of \eqref{lag} is similar to \eqref{staticliquidcrystal} and for the existence of Lorentzian harmonic map with Dirichlet boundary data, the method of \cite{H-K-L} can also be applied. In fact, for any $u\in H^1(M,N)$ and $\psi\in C^{2+\alpha}(\partial M)$, there exists a unique solution  $\Phi(u)$ to the equation \eqref{equation:07} which satisfies $$\|\nabla\Phi(u)\|_{L^2(M)}\leq C(\lambda,\Lambda, \psi, M, N).$$ Then the functional $E_g(u, \Phi(u);M)$ is bounded from below among the class $u\in H^1(M,N)$, and we can deduce the existence of a minimizer by known techniques. Thus, this scheme provides an alternative method for  the existence of Lorentzian harmonic map in a given homotopy class. We skip the details, as the level of difficulty seems to be about the same as for our parabolic-elliptic approach.

\bigskip

{\bf Acknowledgement} {\it Part of this work was done during the last author's visit to the Max Planck Institute for Mathematics in the
Sciences. The author thanks the institute for its hospitality
and good working conditions. }

\end{document}